\documentclass[reqno]{amsproc}%
\usepackage{amsfonts}
\usepackage{amsmath}
\usepackage{amssymb,mathrsfs}
\usepackage[bookmarks]{hyperref}
\usepackage{graphicx}
\usepackage{fancyhdr}
\usepackage{amssymb}%
\providecommand{\U}[1]{\protect\rule{.1in}{.1in}}

\theoremstyle{plain}
\newtheorem{theorem}{Theorem}[section]
\newtheorem{lemma}{Lemma}[section]
\newtheorem{proposition}{Proposition}[section]
\newtheorem{corollary}{Corollary}[section]
\theoremstyle{definition}
\newtheorem{definition}{Definition}[section]

\theoremstyle{remark}
\newtheorem{remark}{Remark}[section]
\newtheorem*{ack}{Acknowledgements}
\numberwithin{equation}{section}

\begin{document}
\title{Generalizations of Rodrigues Type Formulas for Hypergeometric Difference
Equations on Nonuniform Lattices}
\author[J. Cheng ]{Jinfa Cheng}
\address{Jinfa Cheng \\
School of Mathematical Science, Xiamen University, Xiamen, Fujian 361005, China}
\email{jfcheng@xmu.edu.cn}
\author[L. Jia]{Lukun Jia}
\address{Lukun Jia \\
School of Mathematical Science, Xiamen University, Xiamen, Fujian 361005, China}
\email{jialunkun2005@163.com}

\begin{abstract}
By building a second order adjoint difference equations on nonuniform
lattices, the generalized Rodrigues type representation for the second kind
solution of a second order difference equation of hypergeometric type on
nonuniform lattices is given. The general solution of the equation in the form
of a combination of a standard Rodrigues formula and a "generalized" Rodrigues
formula is also established.

\end{abstract}
\date{Nov 4 2018}
\maketitle

\textbf{Keywords:} Special function; Orthogonal polynomials; Adjoint
difference equation; Difference equation of hypergeometric type; Non-uniform lattice

\textbf{MSC 2010:} 33D20,33D45, 33C45.

\section{Introduction}

The special functions of mathematical physics, namely, the classical
orthogonal polynomials and the hypergeometric and cylindrical functions, are
solutions of a differential equation of hypergeometric type.

Let $\sigma(x)$ and $\tau(x)$ be polynomials of at most second and first
degree, respectively, and $\lambda$ be a constant. The following second order
differential equation
\begin{equation}
\sigma(x)y^{\prime\prime}(x)+\tau(x)y^{\prime}(x)+\lambda y(x)=0,
\label{hyper00}%
\end{equation}
is called a \emph{differential equation of hypergeometric type}. If for some
positive integer $n$,
\[
\lambda=\lambda_{n}:=-\frac{n(n-1)\sigma^{\prime\prime}}{2}-n\tau^{\prime
}~\text{and}~\lambda_{m}\neq\lambda_{n}~\text{for}~m=0,1,\dots,n-1,
\]
the equation (\ref{hyper00}) has a polynomial solution $y_{n}(x)$ of degree
$n$ expressed by the Rodrigues formula
\cite{ince1944,hille1976,vicente1943,erdelyi1953,horner1963,horner1964,koekoek2010,nikiforov1983,nikiforov1988,nikiforov1991,wang1989}%
:
\[
y_{n}(x)=\frac{1}{\rho(x)}\frac{d^{n}}{dx^{n}}(\rho(x)\sigma^{n}(x)),
\]
where $\rho(x)$ satisfies the Pearson equation
\[
(\sigma(x)\rho(x))^{\prime}=\tau(x)\rho(x).
\]

These solution functions are useful in quantum mechanics, the theory of group
representations, and computational mathematics. Because of this, the classical
theory of hypergeometric type equations has been greatly developed by G.
Andrews, R. Askey \cite{andrews1985, andrews1999}, J. A. Wilson, M. Ismail
\cite{askey1979, askey1984, askey1985,ismail1989}; F. Nikiforov, K. Suslov, B.
Uvarov, N. M. Atakishiyev \cite{nikiforov1983, nikiforov1988,
nikiforov1991,suslov1989, suslov1992,atakishyev1995}; G. George, M. Rahman
\cite{gasper2004}; T. H. Koornwinder \cite{koornwinder1994}; and many other
researchers like R. \'{A}lvarez-Nodarse, K. L. Cardoso, I. Area, E. Godoy, A.
Ronveaux, A. Zarzo, W. Robin, T. Dreyfus, V. Kac, P. Cheung, and L. K. Jia, J.
F. Cheng, Z.S, Feng [23-32].

Knowing one polynomial solution of (\ref{hyper00}) it is possible to build a
linearly independent solution in many ways: Such as the technique of variation
of constants \cite{ince1944}, by an integral representation using the Cauchy
integral \cite{erdelyi1953,nikiforov1988}. However, Area et al. in
\cite{area2003} first gave an extended Rodrigues type representation of the
second solution of (\ref{hyper00}), they gave an extension of Rodrigues
formula as
\[
y_{n}(x)=\frac{C_{1}}{\rho(x)}\frac{d^{n}}{dx^{n}}\left(  \rho(x)\sigma
^{n}(x)\right)  +\frac{C_{2}}{\rho(x)}\frac{d^{n}}{dx^{n}}\left(
\rho(x)\sigma^{n}(x)\int\frac{dx}{\rho(x)\sigma^{n+1}(x)}\right)  ,
\]
where $C_{1},C_{2}$ are arbitrary constants.

Recently, inspired by \cite{vicente1943}, W. Robin gave a more general
Rodrigues formula \cite{robin2013}
\[
y_{n}(x)=\frac{1}{\rho(x)}\frac{d^{n}}{dx^{n}}\left[  \rho(x)\sigma
^{n}(x)\left(  \int\frac{P_{n}(x)+D_{n}}{\rho(x)\sigma^{n+1}(x)}%
dx+C_{n}\right)  \right]  ,
\]
where $P_{n}(x)$ is an arbitrary polynomial of degree $n$ and $C_{n},D_{n}$ is
an arbitrary constant.

In 1983, the classical theory of hypergeometric type equations had been
greatly developed by Nikiforov, Suslov and Uvarov
\cite{nikiforov1983,nikiforov1988,nikiforov1991} who started from the
following generalization. They replaced (\ref{hyper00}) by a difference
equation on a lattice with variable step size $\nabla x(s)=x(s)-x(s-1)$ :
\begin{equation}
\widetilde{\sigma}[x(s)]\frac{\Delta}{\Delta x(s-1/2)}\left[  \frac{\nabla
y(s)}{\nabla x(s)}\right]  +\frac{1}{2}\widetilde{\tau}[x(s)]\left[
\frac{\Delta y(s)}{\Delta x(s)}+\frac{\nabla y(s)}{\nabla x(s)}\right]
+\lambda y(s)=0. \label{lattice-hyper}%
\end{equation}
Here $\widetilde{\sigma}(x)$ and $\widetilde{\tau}(x)$ are polynomials of at
most the second and first degree in $x(s)$ respectively, $\lambda$ is a
constant,
\[
\Delta y(s)=y(s+1)-y(s),\quad\nabla y(s)=y(s)-y(s-1),
\]
$x(s)$ satisfies
\begin{equation}
\frac{x(s+1)+x(s)}{2}=\alpha x(s+\frac{1}{2})+\beta\label{condition1}%
\end{equation}
(where $\alpha,\beta$ are constants) and
\begin{equation}
\text{ $x^{2}(s+1)+x^{2}(s)$ is a polynomial of degree at most two with
respect to $x(s+\frac{1}{2})$.} \label{condition2}%
\end{equation}

The difference equation (\ref{lattice-hyper}) obtained as a result of
approximating the differential equation (1.1) on a non-uniform lattice is of
independent importance and arises in a number of other questions. Its
solutions essentially generalized the solutions of the original differential
equation and are of interest in their own right. Some of its solutions have
for some time been used in quantum mechanics, the theory of group
representations, and computational mathematics. For more about the difference
equation of hypergeometric type (\ref{lattice-hyper}) on nonuniform lattices,
the reader can see R. Koekoek, P. E. Lesky, R. F. Swarttouw \cite{koekoek2010}
, F. Nikiforov, K. Suslov, B. Uvarov, N. M. Atakishiyev, M. Rahman
\cite{nikiforov1983,nikiforov1988,nikiforov1991,suslov1989,suslov1992,atakishyev1995}%
, A.P. Magnus \cite{magnus1995}, M. Foupouagnigni \cite{foupouagnigni2008,
foupouagnigni2013}, N.S. Witte \cite{witte2011}.

\begin{definition}
Two kinds of lattice functions $x(s)$ are called \emph{nonuniform lattices}
which satisfy the conditions (\ref{condition1}) and (\ref{condition2}):
\begin{align}
\label{lattice1}x(s)=c_{1}q^{s}+c_{2}q^{-s}+c_{3},
\end{align}
\begin{align}
\label{lattice2}x(s)=\widetilde{c}_{1}s^{2}+\widetilde{c}_{2}s+\widetilde{c}%
_{3},
\end{align}
where $c_{i},\widetilde{c}_{i}$ are arbitrary constants and $c_{1}c_{2}\neq0$,
$\widetilde{c}_{1}\widetilde{c}_{2}\neq0$.
\end{definition}

For certain values of $\lambda$ the equation (\ref{lattice-hyper}) has a
polynomial solutions via the difference analog of Rodrigues formula.
Naturally, one can ask whether there is an analog of extension of Rodrigues
formula in the difference equation. As far as we know, the extension to the
case of uniform lattices like $x(s)=s$ and $x(s)=q^{s}$ had already been
obtained in \cite{area2005}. However, in the case of nonuniform lattices
(\ref{lattice1}) or (\ref{lattice2}), it seems more complicated and difficult,
related result in these cases have not been appeared since paper
\cite{area2005} was published in 2005 .

In this paper, we obtain two extensions of Rodrigues formula in the case of
nonuniform lattices (\ref{lattice1}) and (\ref{lattice2}). The paper is
organized as follows. In section 2, calculus on general lattices has been
introduced. In section 3, we first give some Lemmas, then build a second order
adjoint equation in a new way, and an extensions of Rodrigues formula are
given in Theorem \ref{extension1}. In section 4, another more general
Rodrigues formula Theorem \ref{extension2} is estabished in different method
from section 3.

\section{Calculus on general lattices}

Let $x(s)$ be a lattice, where $s\in\mathbb{C}$. For any integer $k$,
$x_{k}(s)=x(s+\frac{k}{2})$ is also a lattice. Given a function $f(s)$, define
two difference operators with respective to $x_{k}(s)$ as
\begin{align}
\label{difference}%
\begin{split}
&  \Delta_{k} f(s)=\frac{\Delta f(s)}{\Delta x_{k}(s)}\\
&  \nabla_{k}f(s)=\frac{\nabla f(s)}{\nabla x_{k}(s)}.
\end{split}
\end{align}
Moreover, for any nonnegative integer $n$, let
\begin{align*}
&  \Delta_{k}^{(n)}f(s)=\left\{
\begin{aligned} &f(s), && \text{ if $n=0$,} \\ &\frac{\Delta}{\Delta x_{k+n-1}(s)}\cdots\frac{\Delta}{\Delta x_{k+1}(s)}\frac{\Delta}{\Delta x_{k}(s)}f(s), && ~\text{if $n\geq1$}. \end{aligned} \right.
\\
&  \nabla_{k}^{(n)}f(s) =\left\{
\begin{aligned} &f(s), && \text{ if $n=0$,} \\ &\frac{\nabla}{\nabla x_{k-n+1}}\cdots\frac{\nabla}{\nabla x_{k-1}(s)}\frac{\nabla}{\nabla x_k(s)}f(s), && ~\text{if $n\geq1$}. \end{aligned} \right.
\end{align*}

The following properties are easy to verify.

\begin{proposition}
\label{product} Given two function $f(s),g(s)$ with complex variable $s$, we
have
\begin{align*}%
\begin{split}
&  \Delta_{k}(f(s)g(s))=f(s+1)\Delta_{k} g(s)+g(s)\Delta_{k} f(s)\\
&  \quad\quad\quad\quad\quad=g(s+1)\Delta_{k} f(s)+f(s)\Delta_{k} g(s),
\end{split}
\\%
\begin{split}
&  \Delta_{k}\left(  \frac{f(s)}{g(s)}\right)  =\frac{g(s+1)\Delta_{k}
f(s)-f(s+1)\Delta_{k} g(s)}{g(s)g(s+1)}\\
&  \quad\quad\quad\quad\quad=\frac{g(s)\Delta_{k} f(s)-f(s)\Delta_{k}
g(s)}{g(s)g(s+1)},
\end{split}
\\%
\begin{split}
&  \nabla_{k}(f(s)g(s))=f(s-1)\nabla_{k} g(s)+g(s)\nabla_{k} f(s)\\
&  \quad\quad\quad\quad\quad=g(s-1)\nabla_{k} f(s)+f(s)\nabla_{k} g(s),
\end{split}
\\%
\begin{split}
&  \nabla_{k}\left(  \frac{f(s)}{g(s)}\right)  =\frac{g(s-1)\nabla_{k}
f(s)-f(s-1)\nabla_{k} g(s)}{g(s)g(s-1)}\\
&  \quad\quad\quad\quad\quad=\frac{g(s)\nabla_{k} f(s)-f(s)\nabla_{k}
g(s)}{g(s)g(s-1)}.
\end{split}
\end{align*}

\end{proposition}

To deal with the inverse of the difference operation $\nabla_{k}$, that is the
integration, following the idea of Gaspard Bangerezako in
\cite{bangerezako2005}, let $\nabla_{k}f(t)=g(t).$ Then
\[
f(t)-f(t-1)=g(t)\left[  x_{k}(t)-x_{k}(t-1)\right]
\]
Choose $N,s\in\mathbb{C}$. Summing from $t=N$ to $t=s$, we have
\[
f(s)-f(N-1)=\sum_{t=N}^{t=s}g(t)\nabla x_{k}(t).
\]
Thus, we define
\begin{align}
\int_{N}^{s}g(t)d_{\nabla}x_{k}(t)=\sum_{t=N}^{t=s}g(t)\nabla x_{k}(t).
\end{align}

It is easy to verify that

\begin{proposition}
\label{integration}\mbox{} Given two function f(s),g(s) with complex variable
N, s, we have \begin{flalign*}
&(1)\nabla_k\left[\int_{N}^{s}g(t)d_\nabla x_k(t)\right]=g(s),&&\\
&(2)\int_{N}^{s}\nabla_kf(t)d_\nabla x_k(t)=f(s)-f(N).&&\\
\end{flalign*}

\end{proposition}

\section{Rodrigues formula}

With notations in section 2, the difference equation of hypergeometric type
(\ref{lattice-hyper}) can be written as
\begin{equation}
\label{operator}\widetilde{\sigma}[x(s)]\Delta_{-1}\nabla_{0}y(s)+\frac
{\widetilde{\tau}[x(s)]}{2}\left[  \Delta_{0}y(s)+\nabla_{0}y(s)\right]
+\lambda y(s)=0.
\end{equation}
In the following context, we assume that the lattice $x(s)$ has the form
(\ref{lattice1}) or (\ref{lattice2}).

Let
\[
z_{k}(s)=\Delta_{0}^{(k)}y(s)=\Delta_{k-1}\Delta_{k-2}\cdots\Delta_{0}y(s).
\]
Then for any nonnegative integer $k$, $z_{k}(s)$ satisfies an equation which
is the same type as (\ref{operator}) \cite{nikiforov1991}:
\begin{equation}
\widetilde{\sigma}_{k}[x_{k}(s)]\Delta_{k-1}\nabla_{k}z_{k}(s)+\frac
{\widetilde{\tau}_{k}[x_{k}(s)]}{2}\left[  \Delta_{k}z_{k}(s)+\nabla_{k}%
z_{k}(s)\right]  +\mu_{k}z_{k}(s)=0, \label{general equation}%
\end{equation}
where $\widetilde{\sigma}_{k}(x_{k})$ and $\widetilde{\tau}_{k}(x_{k})$ are
polynomials of at most second and first degrees in $x_{k}$, respectively,
$\mu_{k}$ is a constant, and%

\begin{align*}
\widetilde{\sigma}_{k}[x_{k}(s)]  &  =\frac{\widetilde{\sigma}_{k-1}%
[x_{k-1}(s+1)]+\widetilde{\sigma}_{k-1}[x_{k-1}(s)]}{2}\\
&  +\frac{1}{4}\Delta_{k-1}\widetilde{\tau}_{k-1}(s)\frac{\Delta
x_{k}(s)+\nabla x_{k}(s)}{2\Delta x_{k-1}(s)}[\Delta x_{k-1}(s)]^{2}\\
&  +\frac{\widetilde{\tau}_{k-1}[x_{k-1}(s+1)]+\widetilde{\tau}_{k-1}%
[x_{k-1}(s)]}{2}\frac{\Delta x_{k}(s)-\nabla x_{k}(s)}{4},\\
\widetilde{\sigma}_{0}[x_{0}(s)]  &  =\widetilde{\sigma}[x(s)];
\end{align*}

\begin{align*}
\widetilde{\tau}_{k}[x_{k}(s)]  &  =\Delta_{k-1}\widetilde{\sigma}%
_{k-1}\left[  x_{k-1}(s)\right]  +\Delta_{k-1}\widetilde{\tau}_{k-1}\left[
x_{k-1}(s)\right]  \frac{\Delta x_{k}(s)-\nabla x_{k}(s)}{4}\\
&  +\frac{\widetilde{\tau}_{k-1}[x_{k-1}(s+1)]+\widetilde{\tau}_{k-1}%
[x_{k-1}(s)]}{2}\frac{\Delta x_{k}(s)+\nabla x_{k}(s)}{2\Delta x_{k-1}(s)},\\
\widetilde{\tau}_{0}[x_{0}(s)]  &  =\widetilde{\tau}[x(s)];
\end{align*}

\[
\mu_{k}=\mu_{k-1}+\Delta_{k-1}\widetilde{\tau}_{k-1}\left[  x_{k-1}(s)\right]
,~\mu_{0}=\lambda.
\]
To study additional properties of solutions of (\ref{general equation}) it is
convenient to use the equation
\[
\frac{1}{2}\left[  \Delta_{k}z_{k}(s)+\nabla_{k}z_{k}(s)\right]  =\Delta
_{k}z_{k}(s)-\frac{1}{2}\Delta\left[  \nabla_{k}z_{k}(s)\right]
\]
and to rewrite (\ref{general equation}) in the equivalent form
\begin{equation}
\sigma_{k}(s)\Delta_{k-1}\nabla_{k}z_{k}(s)+\tau_{k}(s)\Delta_{k}z_{k}%
(s)+\mu_{k}z_{k}(s)=0, \label{operator0}%
\end{equation}
where
\begin{align}
&  \sigma_{k}(s)=\widetilde{\sigma}_{k}[x_{k}(s)]-\frac{1}{2}\widetilde{\tau
}_{k}[x_{k}(s)]\nabla x_{k+1}(s),\label{sigmak0}\\
&  \tau_{k}(s)=\widetilde{\tau}_{k}[x_{k}(s)]. \label{tauk0}%
\end{align}
We can find that
\begin{align}
&  \tau_{k}(s)=\frac{\sigma(s+k)-\sigma(s)+\tau(s+k)\nabla x_{1}(s+k)}{\nabla
x_{k+1}(s)},\label{recurrence1}\\
&  \mu_{k}=\lambda+\sum_{j=0}^{k-1}\Delta_{j}\tau_{j}(s). \label{recurrence2}%
\end{align}

\begin{remark}
When $k$ is a negative integer, we also denote the righthand side of
(\ref{recurrence1}) as $\tau_{k}$.
\end{remark}

Writing the equation (\ref{operator0}) into self-adjoint form:
\[
\Delta_{k-1}\left[  \sigma_{k}(s)\rho_{k}(s)\nabla_{k}z_{k}(s)\right]
+\mu_{k}\rho_{k}(s)z_{k}(s)=0.
\]
Here $\rho_{k}(s)$ satisfies the Pearson type difference equation
\[
\Delta_{k-1}\left[  \sigma_{k}(s)\rho_{k}(s)\right]  =\tau_{k}(s)\rho_{k}(s).
\]
Let $\rho(s)=\rho_{0}(s)$, we can find that
\[
\rho_{k}(s)=\rho(s+k)\prod_{i=1}^{k}\sigma(s+i)
\]
If for a positive integer $n$,
\begin{equation}
\lambda=\lambda_{n}:=-\sum_{j=0}^{n-1}\Delta_{j}\tau_{j}(s),~\text{and}%
~\lambda_{m}\neq\lambda_{n}~\text{for}~m=0,1,\dots,n-1, \label{lambdan}%
\end{equation}
then the equation (\ref{operator}) has a polynomial solution $y_{n}[x(s)]$ of
degree $n$ about $x(s)$. It is expressed by the difference analog of the
Rodrigues formula \cite{nikiforov1983,nikiforov1988,nikiforov1991}:
\[%
\begin{split}
y_{n}[x(s)]  &  =\frac{1}{\rho(s)}\nabla_{n}^{(n)}\left[  \rho_{n}(s)\right]
\\
&  =\frac{1}{\rho(s)}\Delta_{-n}^{(n)}\left[  \rho_{n}(s-n)\right]  .
\end{split}
\]

\section{Explicit form of $\tau_{k} (s),\mu_{k}$ and $\lambda_{n}$}

Now, we deduce the explicit form of $\tau_{k} (s),\mu_{k}$ and $\lambda_{n}$
under the lattice (\ref{lattice1}) and (\ref{lattice2}), respectively.

\begin{proposition}
\label{expresstau} Given any integer $k$, if $x(s)=c_{1}q^{s}+c_{2}%
q^{-s}+c_{3}$, then
\begin{align*}
\tau_{k}(s)  &  =\left[  \frac{q^{k}-q^{-k}}{q^{\frac{1}{2}}-q^{-\frac{1}{2}}%
}\frac{\widetilde{\sigma}^{\prime\prime}}{2}+\left(  q^{k}+q^{-k}\right)
\frac{\widetilde{\tau}^{\prime}}{2}\right]  x_{k}(s)+c(k)\\
&  =[\nu(2k)\frac{{\tilde{\sigma}^{\prime\prime}}}{2}+\alpha(2k)\tilde{\tau
}^{\prime}]x_{k}(s)+c(k)\\
&  =\kappa_{2k+1}x_{k}(s)+c(k),
\end{align*}
where
\[
\nu(\mu)=\left\{
{\begin{array}{*{20}c} {\frac{{q^{\frac{\mu }{2}} - q^{ - \frac{\mu }{2}} }}{{q^{\frac{1}{2}} - q^{ - \frac{1}{2}} }}} \\ \mu \\ \end{array}}%
\right.  ,\alpha(\mu)=\left\{
{\begin{array}{*{20}c} {\frac{{q^{\frac{\mu }{2}} + q^{ - \frac{\mu }{2}} }}{2}} \\ 1 \\ \end{array}}%
\right.  ,
\]
and
\[
\kappa_{\mu}=\alpha(\mu-1)\tilde{\tau}^{\prime}+\nu(\mu-1)\frac{{\tilde
{\sigma}^{\prime\prime}}}{2}.
\]
If $x(s)=\widetilde{c}_{1}s^{2}+\widetilde{c}_{2}s+\widetilde{c}_{3}$, then
\begin{align*}
\tau_{k}(s)  &  =\left[  k\widetilde{\sigma}^{\prime\prime}+\widetilde{\tau
}^{\prime}\right]  x_{k}(s)+\widetilde{c}(k)\\
&  =\kappa_{2k+1}x_{k}(s)+\widetilde{c}(k),
\end{align*}
where $c(k),\widetilde{c}(k)$ are functions with respect to $k$:
\end{proposition}

\begin{align*}
c(k)  &  =c_{3}(1-q^{\frac{k}{2}})(q^{\frac{k}{2}}-q^{-k})+c_{3}%
\frac{(2-q^{\frac{k}{2}}-q^{-\frac{k}{2}})(q^{\frac{k}{2}}-q^{-\frac{k}{2}}%
)}{q^{\frac{1}{2}}-q^{-\frac{1}{2}}}\\
&  +\widetilde{\tau}(0)(q^{\frac{k}{2}}+q^{-\frac{k}{2}})+\widetilde{\sigma
}(0)\frac{q^{\frac{k}{2}}-q^{-\frac{k}{2}}}{q^{\frac{1}{2}}-q^{-\frac{1}{2}}}%
\end{align*}

\[
\widetilde{c}(k)=\frac{\widetilde{\sigma}\prime\prime}{4}\widetilde{c}%
_{1}k^{3}+\frac{3\widetilde{\tau}\prime}{4}\widetilde{c}_{1}k^{2}%
+\widetilde{\sigma}(0)k+2\widetilde{\tau}(0).
\]

\begin{proof}
We only prove the case that $x(s)=c_{1}q^{s}+c_{2}q^{-s}+c_{3}$. By
(\ref{sigmak0}), (\ref{tauk0}) and (\ref{recurrence1}), we have
\begin{equation}
\tau_{k}(s)=\frac{\widetilde{\sigma}[x(s+k)]-\widetilde{\sigma}[x(s)]+\frac
{1}{2}\widetilde{\tau}[x(s+k)]\Delta x(s+k-\frac{1}{2})+\frac{1}%
{2}\widetilde{\tau}[x(s)]\Delta x(s-\frac{1}{2})}{\Delta x_{k-1}(s)}.
\end{equation}
After some simple computations, we obtain
\begin{align*}
&  x(s+k)-x(s)=(q^{\frac{k}{2}}-q^{-\frac{k}{2}})(c_{1}q^{s+\frac{k}{2}}%
-c_{2}q^{-s-\frac{k}{2}}),\\
&  x(s+k)+x(s)=(q^{\frac{k}{2}}+q^{-\frac{k}{2}})x_{k}(s)+c_{3}(2-q^{\frac
{k}{2}}-q^{-\frac{k}{2}}),\\
&  \Delta x_{k-1}(s)=(q^{\frac{1}{2}}-q^{-\frac{1}{2}})(c_{1}q^{s+\frac{k}{2}%
}-c_{2}q^{-s-\frac{k}{2}}).
\end{align*}
Moreover, $\widetilde{\sigma}[x(s)]=\frac{\widetilde{\sigma}^{\prime\prime}%
}{2}x^{2}(s)+\widetilde{\sigma}^{\prime}(0)x(s)+\widetilde{\sigma}(0).$ Then,
\begin{equation}%
\begin{split}
&  \frac{\widetilde{\sigma}[x(s+k)]-\widetilde{\sigma}[x(s)]}{\Delta
x_{k-1}(s)}=\frac{\widetilde{\sigma}^{\prime\prime}}{2}\frac{x^{2}%
(s+k)-x^{2}(s)}{\Delta x_{k-1}(s)}+\widetilde{\sigma}^{\prime}(0)\frac
{x(s+k)-x(s)}{\Delta x_{k-1}(s)}\\
&  =\frac{\widetilde{\sigma}^{\prime\prime}}{2}\frac{q^{k}-q^{-k}}{q^{\frac
{1}{2}}-q^{-\frac{1}{2}}}x_{k}(s)+\widetilde{\sigma}^{\prime}(0)\frac
{q^{\frac{k}{2}}-q^{-\frac{k}{2}}}{q^{\frac{1}{2}}-q^{-\frac{1}{2}}}%
+c_{3}\frac{(2-q^{\frac{k}{2}}-q^{-\frac{k}{2}})(q^{\frac{k}{2}}-q^{-\frac
{k}{2}})}{q^{\frac{k}{2}}-q^{-\frac{k}{2}}}.
\end{split}
\label{widesigmak}%
\end{equation}
Moreover,
\end{proof}

\begin{align}
&  x(s+k)\Delta x(s+k-\frac{1}{2})+x(s)\Delta x(s-\frac{1}{2}%
)\label{doubledelta1}\\
&  =(q^{k}+q^{-k})(q^{\frac{1}{2}}-q^{-\frac{1}{2}})(c_{1}q^{s+\frac{k}{2}%
}-c_{2}q^{-s-\frac{k}{2}})x_{k}(s)\nonumber\\
&  +c_{3}(q^{\frac{1}{2}}-q^{-\frac{1}{2}})(c_{1}q^{s+\frac{k}{2}}%
-c_{2}q^{-s-\frac{k}{2}})(1-q^{\frac{k}{2}})(q^{\frac{k}{2}}-q^{-k}),\nonumber
\end{align}

\begin{equation}
\Delta x(s+k-\frac{1}{2})+\Delta x(s-\frac{1}{2})=(q^{\frac{1}{2}}%
-q^{-\frac{1}{2}})(q^{\frac{k}{2}}+q^{-\frac{k}{2}})(c_{1}q^{s+\frac{k}{2}%
}-c_{2}q^{-s-\frac{k}{2}}), \label{doubledelta2}%
\end{equation}
and $\tau\lbrack x(s)]=\widetilde{\tau}^{\prime}x(s)+\widetilde{\tau}(0). $
Therefore,
\begin{equation}%
\begin{split}
&  \frac{1}{2}\frac{\widetilde{\tau}[x(s+k)]\Delta x(s+k-\frac{1}%
{2})+\widetilde{\tau}[x(s)]\Delta x(s-\frac{1}{2})}{\Delta x_{k-1}(s)}\\
&  =\frac{\widetilde{\tau}^{\prime}}{2}\frac{x(s+k)\Delta x(s+k-\frac{1}%
{2})+x(s)\Delta x(s-\frac{1}{2})}{\Delta x_{k-1}(s)}+\frac{\widetilde{\tau
}(0)}{2}\frac{\Delta x(s+k-\frac{1}{2})+\Delta x(s-\frac{1}{2})}{\Delta
x_{k-1}(s)}\\
&  =\frac{\widetilde{\tau}^{\prime}}{2}(q^{k}+q^{-k})x_{k}(s)+c_{3}%
(1-q^{\frac{k}{2}})(q^{\frac{k}{2}}-q^{-k})+\widetilde{\tau}(0)(q^{\frac{k}%
{2}}+q^{-\frac{k}{2}}).
\end{split}
\label{widetauk}%
\end{equation}
Substituting (\ref{widesigmak}) and (\ref{widetauk}) into (\ref{tauk0}), we
get our conclusion.

\begin{lemma}
(Suslov\cite{suslov1989})\label{lemma0} For $\alpha(\mu),\nu(\mu)$, We have
\[
\sum\limits_{j=0}^{k-1}{\alpha(2j)}=\alpha(k-1)\nu(k),\sum\limits_{j=0}%
^{k-1}{\nu(2j)}=\nu(k-1)\nu(k).
\]

\end{lemma}

From (\ref{recurrence2}) (\ref{lambdan}) and Lemma \ref{lemma0}, we have

\begin{proposition}
\label{expressmu} If $x(s)=c_{1}q^{s}+c_{2}q^{-s}+c_{3}$ or
$x(s)=\widetilde{c}_{1}s^{2}+\widetilde{c}_{2}s+\widetilde{c}_{3}$, then%

\begin{align}
\mu_{k}= \lambda+\kappa_{k} \nu(k).
\end{align}
where
\begin{align*}
\kappa_{k} = \alpha(k - 1)\tilde\tau^{\prime}+ \frac{1}{2}\nu(k -
1)\tilde\sigma^{\prime\prime}.
\end{align*}

\end{proposition}

\begin{proof}
If $x(s)=c_{1}q^{s}+c_{2}q^{-s}+c_{3}$, then%

\begin{align*}
\mu_{k}=\lambda+\sum_{j=0}^{k-1}\left[  \frac{q^{j}-q^{-j}}{q^{\frac{1}{2}%
}+q^{-\frac{1}{2}}}\frac{\widetilde{\sigma}^{\prime\prime}}{2}-(q^{j}%
+q^{-j})\frac{\widetilde{\tau}^{\prime}}{2}\right]  ;
\end{align*}

\begin{align*}%
\begin{array}
[c]{l}%
= \lambda+\sum\limits_{j = 0}^{k - 1} {\nu(2j)} \frac{{\tilde\sigma
^{\prime\prime}}}{2} + \sum\limits_{j = 0}^{k - 1} {\alpha(2j)} \tilde
\tau^{\prime}\\
= \lambda+\nu(k - 1)\nu(k)\frac{{\tilde\sigma^{\prime\prime}}}{2} + \alpha(k -
1)\nu(k)\tilde\tau^{\prime}\\
= \lambda+\kappa_{k} \nu(k)
\end{array}
\end{align*}
where
\begin{align}
\kappa_{k} = \alpha(k - 1)\tilde\tau^{\prime}+ \frac{1}{2}\nu(k -
1)\tilde\sigma^{\prime\prime}.
\end{align}

If $x(s)=\widetilde{c}_{1}s^{2}+\widetilde{c}_{2}s+\widetilde{c}_{3}$, then
\begin{align*}
\mu_{k}  &  = \lambda+\sum_{j=0}^{k-1}\left[  j\widetilde{\sigma}%
^{\prime\prime}+\widetilde{\tau}^{\prime}\right] \\
&  = \lambda+\frac{{(k - 1)k}}{2}\tilde\sigma^{\prime\prime}+ k\tilde
\tau^{\prime}\\
&  = \lambda+\kappa_{k} \nu(k).
\end{align*}

\end{proof}

From Proposition \ref{expressmu} and (\ref{lambdan}), we have
\begin{align}
\lambda_{n} = - n\kappa_{n} .
\end{align}

\section{Adjoint equations}

Let
\begin{align}
\label{operator1}L[y]=\sigma(s)\Delta_{-1}\nabla_{0} y(s)+\tau(s)\Delta
_{0}y(s)+\lambda y(s)=0.
\end{align}
The equation (\ref{operator1}) has a self-adjoint form
\begin{align}
\label{selfadjoint}\Delta_{-1}\left[  \sigma(s)\rho(s)\nabla_{0}y(s)\right]
+\lambda\rho(s)y(s)=0,
\end{align}
where $\rho(s)$ satisfies a Pearson type equation:
\begin{align}
\label{pearson1}\Delta_{-1}\left[  \sigma(s)\rho(s)\right]  =\tau(s)\rho(s).
\end{align}

In order to obtain extension of Rodrigues formula in the nonuniform lattice
case, it is crucial to define and build second order adjoint difference
equations.

Let $w(s)=\rho(s)y(s)$. Then
\begin{align}
\label{nablay}\nabla_{0}y(s)=\nabla_{0}\frac{w(s)}{\rho(s)}=\frac
{\rho(s-1)\nabla_{0}w(s)-w(s-1)\nabla_{0}\rho(s)}{\rho(s)\rho(s-1)}%
\end{align}
Substituting (\ref{nablay}) into (\ref{selfadjoint}), we obtain
\begin{align}
\label{selfadjoint1}\Delta_{-1}\left[  \sigma(s)(\nabla_{0}w(s)-w(s-1)\frac
{\nabla_{0}\rho(s)}{\rho(s-1)})\right]  +\lambda w(s)=0.
\end{align}
By the Pearson type equation (\ref{pearson1}),
\[
\frac{\Delta[\sigma(s)\rho(s)]}{\Delta x_{-1}(s)}=\frac{\sigma(s+1)\Delta
\rho(s)+\Delta\sigma(s)\rho(s)}{\Delta x_{-1}(s)}=\tau(s)\rho(s).
\]
Then
\begin{align}
\label{rho1}\frac{\nabla\rho(s)}{\rho(s-1)}=\frac{\tau(s-1)\nabla
_{-1}(s)-\nabla\sigma(s)}{\sigma(s)}%
\end{align}
Substituting (\ref{rho1}) into (\ref{selfadjoint1}), we have
\begin{align}
\label{adjoint equation}L^{*}[w]:=\sigma^{*}(s)\Delta_{-1}\nabla_{0}
w(s)+\tau^{*}(s)\Delta_{0} w(s)+\lambda^{*}w(s)=0,
\end{align}
where
\begin{align}
&  \sigma^{*}(s)=\sigma(s-1)+\tau(s-1)\nabla x_{-1}(s),\label{sigmastar}\\
&  \tau^{*}(s)=\frac{\sigma(s+1)-\sigma(s-1)}{\Delta x_{-1}(s)}-\tau
(s-1)\frac{\nabla x_{-1}(s)}{\Delta x_{-1}(s)},\label{taustar}\\
&  \lambda^{*}=\lambda-\Delta_{-1}\left(  \tau(s-1)\frac{\nabla x_{-1}%
(s)}{\nabla x(s)}-\frac{\nabla\sigma(s)}{\nabla x(s)}\right)  .
\label{lambdastar}%
\end{align}

\begin{definition}
\label{def} The equation (\ref{adjoint equation}) is called the \emph{adjoint
equation corresponding to} (\ref{operator1}).
\end{definition}

From Definition \ref{def}, it is easy to obtain

\begin{proposition}
\label{pro3} For y(s), we have
\begin{align}
\label{relation}L^{*}[\rho y]=\rho L[y].
\end{align}

\end{proposition}

\begin{lemma}
(Suslov\cite{suslov1989})\label{lemma3} Let $x=x(s)$ be a lattice for which
(\ref{lattice1}) and (\ref{lattice2}) hold, then the functions $\tilde{\sigma
}_{\nu}(s)$ and $\tau_{\nu}(s)$ defined by the equalities
\begin{equation}
\tilde{\sigma}_{\nu}(s)=\sigma(s)+\frac{1}{2}\tau_{\nu}(s)\nabla x_{\nu+1}(s),
\label{sigmanu}%
\end{equation}%
\begin{equation}
\tau_{\nu}(s)\nabla x_{\nu+1}(s)=\sigma(s+\nu)-\sigma(s)+\tau(s+\nu)\nabla
x_{1}(s+\nu), \label{taunu}%
\end{equation}
are polynomials $\tilde{\sigma}_{\nu}(s)=\tilde{\sigma}_{\nu}(x_{\nu}%
),\tau_{\nu}(s)=\tilde{\tau}_{\nu}(x_{\nu})$ of degree as most two and one
respectively in the variable $x_{\nu}=x(s+\frac{\nu}{2}),\nu\in R.$
\end{lemma}

By the use of Proposition \ref{expresstau} and Lemma \ref{lemma3}, it is not
hard to obtain

\begin{corollary}
\label{corollary2} For (\ref{taustar}) and (\ref{lambdastar}),we have
\begin{align}
\label{tau*}\tau^{*} (s) = - \tau_{ - 2} (s + 1) = [\gamma(4)\frac
{{\tilde\sigma^{\prime\prime}}}{2} - \alpha(4)\tilde\tau^{\prime}]x_{0} (s) +
c( - 2),
\end{align}
and
\begin{align}
\label{lambda*}\lambda^{*} = \lambda- \Delta_{ - 1} \tau_{ - 1} (s) = \lambda-
\gamma(2)\frac{{\tilde\sigma^{\prime\prime}}}{2} - \alpha(2)\tilde\tau
^{\prime}=\lambda- \kappa_{ - 1} .
\end{align}

\end{corollary}

\begin{proof}
Since
\[
\sigma(s - 1) - \sigma(s + 1) + \tau(s - 1)\nabla x_{ - 1} (s) = \sigma(s - 1)
- \sigma(s + 1) + \tau(s - 1)\nabla x_{1} (s - 1).
\]
Set $s+1=z$, then by (\ref{taunu}), we have
\[
\sigma(s - 1) - \sigma(s + 1) + \tau(s - 1)\nabla x_{1} (s - 1) = \sigma(z -
2) - \sigma(z) + \tau(z - 2)\nabla x_{1} (z - 2)
\]
\[
= \tau_{ - 2} (z)\nabla x_{ - 1} (z) = \tau_{ - 2} (s + 1)\nabla x_{ - 1} (s +
1) = \tau_{ - 2} (s + 1)\Delta x_{ - 1} (s),
\]
Now from (\ref{taustar}) and Proposition \ref{expresstau}, one has
\[
\tau^{*} (s) = -\tau_{ - 2} (s + 1) = -[\gamma( - 4)\frac{{\tilde
\sigma^{\prime\prime}}}{2} + \alpha( - 4)\tilde\tau^{\prime}]x_{ - 2} (s + 1)
+ c( - 2)
\]
\[
= [\gamma(4)\frac{{\tilde\sigma^{\prime\prime}}}{2} - \alpha(4)\tilde
\tau^{\prime}]x_{0} (s) + c( - 2).
\]
In the same way, from (\ref{lambdastar}) and (\ref{taunu}), we obtain
\[
\lambda^{*} = \lambda- \Delta_{ - 1} \tau_{ - 1} (s) = \lambda- \Delta_{ - 1}
\{ [\gamma(-2)\frac{{\tilde\sigma^{\prime\prime}}}{2} + \alpha(-2)\tilde
\tau]x_{ - 1} (s)\}
\]
\[
= \lambda+ \gamma(2)\frac{{\tilde\sigma^{\prime\prime}}}{2} - \alpha
(2)\tilde\tau^{\prime}=\lambda- \kappa_{ - 1} .
\]

\end{proof}

With respect to the adjoint equation (\ref{adjoint equation}), we find that it
has the following interesting dual properties

\begin{proposition}
\label{pro33} For adjoint equation (\ref{adjoint equation}), we have
\begin{align}
&  \sigma(s)=\sigma^{*}(s-1)+\tau^{*}(s-1)\nabla x_{-1}(s),\label{sigma}\\
&  \tau(s)=\frac{\sigma^{*}(s+1)-\sigma^{*}(s-1)}{\Delta x_{-1}(s)}-\tau
^{*}(s-1)\frac{\nabla x_{-1}(s)}{\Delta x_{-1}(s)},\label{tau}\\
&  \lambda=\lambda^{*}-\Delta_{-1}\left(  \tau^{*}(s-1)\frac{\nabla x_{-1}%
(s)}{\nabla x(s)}-\frac{\nabla\sigma^{*}(s)}{\nabla x(s)}\right)  .
\label{lambda}%
\end{align}

\end{proposition}

\begin{proof}
From (\ref{taustar}) we have
\begin{align}
\tau^{*} (s)\Delta x_{ - 1} (s) = \sigma(s + 1) - \sigma(s - 1) - \tau(s -
1)\nabla x_{ - 1} (s), \label{taustar2}%
\end{align}
from (\ref{sigmastar}) and (\ref{taustar2}), we have
\[
\sigma(s + 1) = \sigma^{*} (s) + \tau^{*} (s)\Delta x_{ - 1} (s),
\]
therefore
\[
\sigma(s) = \sigma^{*} (s - 1) + \tau^{*} (s - 1)\nabla x_{ - 1} (s).
\]
By (\ref{sigmastar}), we obtain
\[
\tau(s - 1) = \frac{{\sigma^{*} (s) - \sigma(s - 1)}}{{\nabla x_{ - 1} (s)}} =
\frac{{\sigma^{*} (s) - \sigma^{*} (s - 2) - \tau^{*} (s - 2)\nabla x_{ - 1}
(s - 1)}}{{\nabla x_{ - 1} (s)}},
\]
therefore
\[
\tau(s) = \frac{{\sigma^{*} (s\mathrm{\ + }1) - \sigma^{*} (s - 1) - \tau^{*}
(s - 1)\nabla x_{ - 1} (s)}}{{\Delta x_{ - 1} (s)}}.
\]
Moreover
\begin{align}%
\begin{split}
\tau(s - 1)\nabla x_{ - 1} (s) - \nabla\sigma(s)  &  = \sigma^{*} (s) -
\sigma(s - 1) - \nabla\sigma(s) = \sigma^{*} (s) - \sigma(s),\\
&  = \sigma^{*} (s) - [\sigma^{*} (s - 1) + \tau^{*} (s - 1)\nabla x_{ - 1}
(s)]\\
&  = \nabla\sigma^{*} (s) - \tau^{*} (s - 1)\nabla x_{ - 1} (s),
\end{split}
\label{tau2}%
\end{align}
therefore, from (\ref{lambdastar} and (\ref{tau2}), one has
\[
\lambda= \lambda^{*} + \Delta_{ - 1} [\frac{{\tau(s - 1)\nabla x_{ - 1} (s) -
\nabla\sigma(s)}}{{\nabla x(s)}}] = \lambda^{*} - \Delta_{ - 1} [\frac
{{\tau^{*} (s - 1)\nabla x_{ - 1} (s) - \nabla\sigma^{*} (s)}}{{\nabla x(s)}%
}].
\]

\end{proof}

In the same way as Corollary \ref{corollary2}, we have

\begin{corollary}
\label{corollary3} For (\ref{tau})and (\ref{lambda}),we have
\begin{align}
\tau(s) = - \tau^{*} _{ - 2} (s + 1),
\end{align}
and
\begin{align}
\lambda=\lambda^{*} - \kappa^{*} _{ - 1} .
\end{align}

\end{corollary}

\begin{proposition}
The adjoint equation (\ref{adjoint equation}) can be rewritten as
\begin{align}
\sigma(s + 1)\Delta_{ - 1} \nabla_{0} w(s) - \tau_{ - 2} (s + 1)\nabla_{0}
w(s) + (\lambda- \kappa_{ - 1} ) w(s) = 0. \label{adjoint2}%
\end{align}

\end{proposition}

\begin{proof}
Since%
\[
\Delta_{0} w(s) - \nabla_{0} w(s) = \Delta(\frac{{\nabla w(s)}}{{\nabla x(s)}%
}),
\]
we have
\[
\tau^{*} (s)\Delta_{0} w(s) = \tau^{*} (s)\nabla_{0} w(s) + \tau^{*}
(s)\Delta(\frac{{\nabla w(s)}}{{\nabla x(s)}})
\]
\begin{align}
\label{equ}= \tau^{*} (s)\nabla_{0} w(s) + \tau^{*} (s)\Delta x_{ - 1}
(s)\frac{\Delta}{{\Delta x_{ - 1} (s)}}(\frac{{\nabla w(s)}}{{\nabla x(s)}}),
\end{align}
substituting (\ref{equ}) into (\ref{adjoint equation}), we obtain
\begin{align}
\label{equ2}[\sigma^{*} (s) + \tau^{*} (s)\Delta x_{ - 1} (s)]\frac{\Delta
}{{\Delta x_{ - 1} (s)}}(\frac{{\nabla w(s)}}{{\nabla x(s)}}) + \tau^{*}
(s)\nabla_{0} w(s) + \lambda^{*} w(s) = 0.
\end{align}
from (\ref{sigma}), one has
\begin{align}
\label{equ3}\sigma^{*} (s) + \tau^{*} (s)\Delta x_{ - 1} (s) = \sigma(s + 1).
\end{align}
Substituting (\ref{equ3}) into (\ref{equ2}), and by (\ref{tau*}), we obtain
\[
\sigma(s + 1)\Delta_{ - 1} \nabla_{0} w(s) - \tau_{ - 2} (s + 1)\nabla_{0}
w(s) + (\lambda- \kappa_{ - 1}) w(s) = 0.
\]

\end{proof}

Next we will prove that the adjoint equation (\ref{adjoint equation}) or
(\ref{adjoint2}) is also hypergeometric type difference equations on
nonuniform lattices. It only needs to prove
\[
\tilde\sigma^{*} (s) = \sigma^{*} (s) + \frac{1}{2}\tau^{*} (s)\Delta x_{ - 1}
(s) = \sigma(s + 1) +\frac{1}{2}\tau_{ - 2} (s + 1)\Delta x_{ - 1} (s)
\]
is polynomial of degree at most two in the variable $x_{0}(s)$.

In fact, from Lemma \ref{lemma3} and (\ref{sigmanu}), one has
\[
\tilde\sigma^{*} (s) = \sigma(s + 1) + \frac{1}{2}\tau_{ - 2} (s + 1)\nabla
x_{ - 1} (s + 1) = \tilde\sigma_{ - 2} (s + 1)
\]
is polynomial of degree at most two in the variable $x_{-2}(s+1)=x_{0}(s)$.

Therefore, we have

\begin{theorem}
\label{theorem1} The adjoint equation (\ref{adjoint2}) or
\begin{align}
\label{theo}\tilde\sigma_{ - 2} (s + 1)\Delta_{ - 1} \nabla_{0} w(s) -
\frac{1}{2}\tau_{ - 2} (s + 1)[\Delta_{0} w(s) + \nabla_{0} w(s)] + (\lambda-
\kappa_{ - 1} ) w(s) = 0
\end{align}
is also hypergeometric type difference equations on nonuniform lattices.
\end{theorem}

\section{An extension of Rodrigues formula}

Let
\[
Y_{n}(s)=\rho_{n}(s-n)=\rho(s)\prod_{j=0}^{n-1}\sigma(s-i).
\]
We now construct a difference equation of the form (\ref{operator0}) which
satisfied by $Y_{n}(s)$. Rewrite $Y_{n}(s)=\rho(s)\sigma(s)\prod_{j=1}%
^{n-1}\sigma(s-i)$. Then using Proposition \ref{product} and the Pearson type
equation (\ref{pearson1}), we have
\begin{align*}%
\begin{split}
\nabla_{-n}Y_{n}(s)  &  =\frac{\nabla Y_{n}(s)}{\nabla x_{-n}(s)}\\
&  =\frac{1}{\nabla x_{-n}(s)}\left[  \rho(s)\sigma(s)\prod_{j=1}^{n-1}%
\sigma(s-i)-\rho(s-1)\sigma(s-1)\prod_{i=1}^{n-1}\sigma(s-1-i)\right] \\
&  =\frac{1}{\nabla x_{-n}(s)}\Big\{\lbrack\tau(s-1)\rho(s-1)\nabla
x_{-1}(s)+\sigma(s-1)\rho(s-1)]\prod_{i=1}^{n-1}\sigma(s-i)\\
&  -\rho(s-1)\sigma(s-1)\prod_{i=1}^{n-1}\sigma(s-1-i)\Big\}.
\end{split}
\end{align*}
Then we have a difference equation satisfied by $Y_{n}(s)$:
\begin{align}
\label{Yn1}\sigma(s-n)\nabla_{-n}Y_{n}(s)=\left(  \frac{\sigma(s-1)-\sigma
(s-n)}{\nabla x_{-n}(s)}+\tau(s-1)\frac{\nabla x_{-1}(s)}{\nabla x_{-n}%
(s)}\right)  Y_{n}(s-1).
\end{align}

\begin{proposition}
\label{the other} If $u_{1}(s)$ is a nontrivial solution of the difference
equation
\begin{align}
\label{first}p_{1}(s)\nabla_{k} u(s)=p_{0}(s)u(s-1),
\end{align}
with $p_{1}(s)\neq0$, then it satisfies a difference equation
\begin{align}
\label{second}%
\begin{split}
&  \left(  p_{1}(s)-p_{0}(s)\nabla x_{k}(s)\right)  \Delta_{k-1}\nabla_{k}
u(s)\\
&  +\left(  \Delta_{k-1}p_{1}(s)-p_{0}(s)\frac{\nabla x_{k}(s)}{\Delta
x_{k-1}(s)}\right)  \Delta_{k}u(s)-\Delta_{k-1}p_{0}(s)u(s)=0
\end{split}
\end{align}
Furthermore, the other solution of (\ref{second}) is
\begin{align*}
u_{2}(s)=Cu_{1}(s)\int_{N}^{s} \frac{1}{p_{1}(t)u_{1}(t)}d_{\nabla}x_{k}(t),
\end{align*}
where $C$ is a constant.
\end{proposition}

\begin{proof}
Applying the operator $\Delta_{k-1}$ to both sides of (\ref{first}), we have
\begin{align}
\label{ukk}p_{1}(s)\Delta_{k-1}\nabla_{k} u(s)+\Delta_{k-1}p_{1}(s)\Delta_{k}
u(s)=u(s)\Delta_{k-1}p_{0}(s)+p_{0}(s)\Delta_{k-1}u(s-1).
\end{align}
Since
\[
\Delta_{k-1}\nabla_{k} u(s)=\frac{1}{\Delta x_{k-1}(s)}\left(  \Delta
_{k}u(s)-\frac{\nabla u(s)}{\nabla x_{k}(s)}\right)  ,
\]
we have
\begin{align}
\label{uk-1}\Delta_{k-1}u(s-1)=\frac{\nabla x_{k}(s)}{\Delta x_{k-1}%
(s)}\left(  \Delta_{k}u(s)-\Delta_{k-1}\nabla_{k} u(s)\Delta x_{k-1}%
(s)\right)  .
\end{align}
Substituting (\ref{uk-1}) into (\ref{ukk}), we obtain the equation
(\ref{second}) about $u(s)$. Denote the other solution of (\ref{second}) as
$u_{2}(s)$, then
\begin{align}
\label{third}\nabla_{k}\left(  \frac{u_{2}(s)}{u_{1}(s)}\right)  =\frac
{u_{1}(s-1)\nabla_{k} u_{2}(s)-u_{2}(s-1)\nabla_{k} u_{1}(s)}{u_{1}%
(s)u_{1}(s-1)}.
\end{align}
By the above deduction, $u_{2}(s)$ satisfies
\[
\Delta_{k-1}[p_{1}(s)\nabla_{k} u(s)-p_{0}(s)u(s-1)]=0.
\]
Thus for any constant $C$,
\[
p_{1}(s)\nabla_{k}u_{2}(s)-p_{0}(s)u_{2}(s-1)=C.
\]
Then, substituting (\ref{first}) into (\ref{third}), we have
\begin{align*}
\nabla_{k}\left(  \frac{u_{2}(s)}{u_{1}(s)}\right)  =\frac{C}{p_{1}%
(s)u_{1}(s)}.
\end{align*}
By Proposition \ref{integration},
\begin{align*}
u_{2}(s)=Cu_{1}(s)\int_{N}^{s}\frac{1}{p_{1}(t)u_{1}(t)}d_{\nabla}x_{k}(t).
\end{align*}

\end{proof}

By the above proposition, we obtain%
\begin{align}
\label{Yns}\widehat{\sigma}(s)\Delta_{-(n+1)}\nabla_{-n} Y_{n}%
(s)+\widehat{\tau}(s)\Delta_{-n}Y_{n}(s)+\widehat{\lambda}Y_{n}(s)=0,
\end{align}
where
\begin{align}
&  \widehat{\sigma}(s)=\sigma(s-1)+\tau(s-1)\nabla x_{-1}(s)=\sigma
^{*}(s),\label{hatsigma}\\
&  \widehat{\tau}(s)=\frac{\sigma(s-n+1)-\sigma(s-1)}{\Delta x_{-(n+1)}%
(s)}-\tau(s-1)\frac{\nabla x_{-1}(s)}{\Delta x_{-(n+1)}(s)}= - \tau_{ - (n +
2)} (s + 1),\label{hattau}\\
&  \widehat{\lambda}=-\Delta_{-(n+1)}\left(  \frac{\sigma(s-1)-\sigma
(s-n)}{\nabla x_{-n}(s)}+\tau(s-1)\frac{\nabla x_{-1}(s)}{\nabla x_{-n}%
(s)}\right)  = - \Delta_{ - (n + 1)} \tau_{ - (n + 1)} (s). \label{hatlambda}%
\end{align}

By the same way as section 5, the equation (\ref{Yns}) can be also rewritten
as
\begin{align}
\sigma(s + 1)\Delta_{ - (n + 1)} \nabla_{ - n} Y_{n} (s) - \tau_{ - (n + 2)}
(s + 1)\nabla_{ - n} Y_{n} (s) +\widehat{\lambda} Y_{n} (s) = 0.
\label{equation61}%
\end{align}

Denote the other solution of (\ref{Yns}) as $\widehat{Y}_{n}(s)$. Then
\begin{align}
\label{hatYn}\widehat{Y}_{n}(s)=\rho(s)\prod_{j=0}^{n-1}\sigma(s-j)\int%
_{N}^{s}\frac{1}{\rho(t)\prod_{j=0}^{n}\sigma(t-j)}d_{\nabla}x_{-n}(t).
\end{align}

Now, for the equation (\ref{Yns}), let $Y_{n}^{(n)}(s)=\Delta_{-n}^{(n)}%
Y_{n}(s)$. Then $Y_{n}^{(n)}(s)$ satisfies
\begin{align}
\label{Ynsn}\widehat{\sigma}(s)\Delta_{-1}\nabla_{0}Y_{n}^{(n)}%
(s)+\widehat{\tau}_{n}(s)\Delta_{0}Y_{n}^{(n)}(s)+\widehat{\mu}_{n}Y_{n}%
^{(n)}(s)=0.
\end{align}
By the recurrence relation (\ref{recurrence1}) and (\ref{recurrence2}), for
any nonnegative integer $k$, we have
\begin{align}
\label{taukk}%
\begin{split}
\widehat{\tau}_{k}(s)  &  =\frac{\widehat{\sigma}(s+k)-\widehat{\sigma
}(s)+\widehat{\tau}(s+k)\Delta x_{-n}(s+k-\frac{1}{2})}{\Delta x_{-n}%
(s+\frac{k-1}{2})}\\
&  =\frac{\sigma(s-n+k+1)-\sigma(s-1)-\tau(s-1)\nabla x_{-1}(s)}{\Delta
x_{-n}(s+\frac{k-1}{2})}\\
&  = - \tau_{ - (n - k + 2)} (s + 1),
\end{split}
\end{align}
and
\begin{align}
\label{muk}\hat\mu_{n} = \hat\lambda+ \sum\limits_{k = 0}^{n - 1} {\Delta_{k -
n} \hat\tau_{k} (s)}.
\end{align}

When $k$ is a negative integer, we also denote the righthand sides of
(\ref{taukk}) as $\widehat{\tau}_{k}(s)$. When $k=n$, we obtain
\begin{align*}
\widehat{\tau}_{n}(s)=\frac{\sigma(s+1)-\sigma(s-1)}{\Delta x(s-\frac{1}{2}%
)}-\tau(s-1)\frac{\nabla x_{-1}(s)}{\Delta x_{-1}(s)} = - \tau_{ - 2} (s +
1)=\tau^{*}(s).
\end{align*}
By the same way as section 5, the equation (\ref{Ynsn}) can also be rewritten
as
\begin{align}
\label{Ynsn62}\sigma(s + 1)\Delta_{ - 1} \nabla_{0} Y_{n} (s) - \tau_{ - 2} (s
+ 1)\nabla_{0} Y_{n} (s) + \hat\mu_{n} Y_{n} (s) = 0.
\end{align}

For the computation of $\widehat{\mu}_{n}$ we need a result whose proof is
similar to Proposition \ref{expresstau}.

\begin{lemma}
\label{lattice} Given any integer $k$, if $x(s)=c_{1}q^{s}+c_{2}q^{-s}+c_{3}$,
then
\begin{align*}%
\begin{split}
\widehat{\tau}_{k}(s)  &  =\left[  \frac{q^{k-n+2}-q^{n-k-2}}{q^{\frac{1}{2}%
}-q^{-\frac{1}{2}}}\frac{\widetilde{\sigma}^{\prime\prime}}{2}-\left(
q^{k-n+2}+q^{n-k-2}\right)  \frac{\widetilde{\tau}^{\prime}}{2}\right]
x_{k-n}(s)+\widehat{c}_{1}(k);\\
&  = \{ \nu[ - 2(n - k - 2)]\frac{{\tilde\sigma^{\prime\prime}}}{2} - \alpha[
- 2(n - k - 2)]\tau^{\prime}\} x_{n - k} (s) + \hat c_{1} (k)\\
&  = - \kappa_{2(n - k - 2) + 1}x_{k - n} + \hat c_{1} (k);
\end{split}
\end{align*}

if $x(s)=\widetilde{c}_{1}s^{2}+\widetilde{c}_{2}s+\widetilde{c}_{3}$, then
\begin{align*}%
\begin{split}
\widehat{\tau}_{k}(s)  &  =\left[  (k-n+2)\widetilde{\sigma}^{\prime\prime
}-\widetilde{\tau}^{\prime}\right]  x_{k-n}(s)+\widehat{c}_{2}(k)\\
&  = - \kappa_{2(n - k - 2) + 1}x_{k - n} + \hat c_{2} (k);
\end{split}
\end{align*}
where $\widehat{c}_{1}(k),\widehat{c}_{2}(k)$ are functions with respect to
$k$:%

\begin{align*}%
\begin{split}
&  \widehat{c}_{1}(k)=\frac{q^{\frac{k-n+2}{2}}-q^{\frac{n-k-2}{2}}}%
{q^{\frac{1}{2}}-q^{-\frac{1}{2}}}[\widetilde{\sigma}^{\prime}(0)+c_{3}%
(2-q^{\frac{k-n+2}{2}}-q^{\frac{n-k-2}{2}})]\\
&  +\widetilde{\tau}(0)(q^{\frac{k-n+2}{2}}+q^{\frac{n-k-2}{2}})+c_{3}%
\widetilde{\tau}^{\prime\frac{k-n+2}{2}})(q^{\frac{k-n+2}{2}}-q^{n-k-2}),
\end{split}
\\
&  \widehat{c}_{2}(k)=\widetilde{\sigma}'(0)(k-n+2)+\frac{\widetilde{\sigma
}''}{4}\widetilde{c}_{1}(k-n+2)^{3}+\frac{3\widetilde{\tau}'}{4}%
\widetilde{c}_{1}(k-n+2)^{2}+2\widetilde{\tau}(0).
\end{align*}

\end{lemma}

\begin{corollary}
\label{corollary6} If $x(s)=c_{1}q^{s}+c_{2}q^{-s}+c_{3}$, or
$x(s)=\widetilde{c}_{1}s^{2}+\widetilde{c}_{2}s+\widetilde{c}_{3}$, then
\begin{align*}
\widehat{\mu}_{n}=- \kappa_{ - 1} - \kappa_{n} \nu(n)=- \kappa_{n-1} \nu(n+1).
\end{align*}

\end{corollary}

\begin{proof}
By (\ref{hatlambda}), we can see that $\hat\lambda\mathrm{\ = }\Delta_{ - (n +
1)} \tau_{ - (n + 1)} (s) = \Delta_{ - (n + 1)} \hat\tau_{ - 1} (s)$. Then, we
can rewrite (\ref{muk}) as
\begin{align}
\label{muhat}\widehat{\mu}_{n}={\hat\lambda}+\sum_{k=0}^{n-1}\Delta
_{k-n}\widehat{\tau}_{k}(s)=\sum_{k=-1}^{n-1}\Delta_{k-n}\widehat{\tau}%
_{k}(s).
\end{align}
Then from Lemma \ref{lattice}, if $x(s)=c_{1}q^{s}+c_{2}q^{-s}+c_{3}$,
\begin{align*}%
\begin{split}
\widehat{\mu}_{n}  &  =\sum_{k=-1}^{n-1}\left[  \frac{q^{k-n+2}-q^{n-k-2}%
}{q^{\frac{1}{2}}-q^{\frac{1}{2}}}\frac{\sigma^{\prime\prime}}{2}-\left(
q^{k-n+2}+q^{n-k-2}\right)  \frac{\tau^{\prime}}{2}\right] \\
&  =\sum_{k=-1}^{n-1}\left[  \frac{q^{-k}-q^{k}}{q^{\frac{1}{2}}-q^{\frac
{1}{2}}}\frac{\sigma^{\prime\prime}}{2}-\left(  q^{-k}+q^{k}\right)
\frac{\tau^{\prime}}{2}\right] \\
&  =- \kappa_{ - 1}+ \sum_{k=0}^{n-1}\left[  \frac{q^{-k}-q^{k}}{q^{\frac
{1}{2}}-q^{\frac{1}{2}}}\frac{\sigma^{\prime\prime}}{2}-\left(  q^{-k}%
+q^{k}\right)  \frac{\tau^{\prime}}{2}\right] \\
&  = - \kappa_{ - 1} - \kappa_{n} \nu(n)=- \kappa_{n-1} \nu(n+1).
\end{split}
\end{align*}
If $x(s)=\widetilde{c}_{1}s^{2}+\widetilde{c}_{2}s+\widetilde{c}_{3}$,
\begin{align*}%
\begin{split}
\widehat{\mu}_{n}  &  =\sum_{k=-1}^{n-1}\left[  (k-n+2)\sigma^{\prime\prime
}-\tau^{\prime}\right] \\
&  =\sum_{k=-1}^{n-1}\left[  -k\sigma^{\prime\prime}-\tau^{\prime}\right] \\
&  =- \kappa_{ - 1}+\sum_{k=0}^{n-1}\left[  -k\sigma^{\prime\prime}%
-\tau^{\prime}\right] \\
&  = - \kappa_{ - 1} - n\kappa_{n}=- (n+1)\kappa_{n-1}.
\end{split}
\end{align*}

\end{proof}

\begin{theorem}
If $\lambda= \lambda_{n} = - \kappa_{n} \nu(n)$, then the equation
(\ref{Ynsn}) is
\[
L^{*}[Y_{n}^{(n)}(s)]=0.
\]

\end{theorem}

\begin{proof}
We only need to prove that when $\lambda=\lambda_{n}$, we have $\widehat{\mu
}_{n}=\lambda^{*}$. By (\ref{lambdastar}) and Proposition \ref{expresstau}, we
have
\begin{align}
\label{lambdastar00}%
\begin{split}
\lambda^{*}  &  =\lambda_{n}-\Delta_{-1}\left(  \tau(s-1)\frac{\nabla
x_{-1}(s)}{\nabla x(s)}-\frac{\nabla\sigma(s)}{\nabla x(s)}\right) \\
&  =\lambda_{n}-\Delta_{-1}\tau_{-1}(s)\\
&  = - \kappa_{n} \nu(n) - \kappa_{ - 1}=- \kappa_{n-1} \nu(n+1) .
\end{split}
\end{align}

Therefore, By Corollary \ref{corollary6}, we obtain $\widehat{\mu}_{n}%
=\lambda^{*}$.
\end{proof}

By the identity (\ref{relation}), we have $L[\frac{1}{\rho}Y_{n}^{(n)}%
]=L^{*}[Y_{n}^{(n)}]=0$. Then we get an extension of Rodrigues type formula:

\begin{theorem}
\label{extension1} If
\[
\lambda=\lambda_{n} ~\text{and}~ \lambda_{m}\neq\lambda_{n}~\text{for}%
~m=0,1,\dots,n-1,
\]
then the general solution of (\ref{operator}) is
\begin{align*}%
\begin{split}
y_{n}(s)  &  =\frac{C_{1}}{\rho(x)}\Delta_{-n}^{(n)}[\rho(s)\prod_{j=0}%
^{n-1}\sigma(s-j)]\\
&  +\frac{C_{2}}{\rho(x)}\Delta_{-n}^{(n)}\left[  \rho(s)\prod_{j=0}%
^{n-1}\sigma(s-j)\int_{N}^{s}\frac{1}{\rho(t)\prod_{j=0}^{n}\sigma
(t-j)}d_{\nabla}x_{-n}(t)\right]  .
\end{split}
\end{align*}

\end{theorem}


\section{More general Rodrigues formula}

\begin{proposition}
\cite[p.62]{nikiforov1991}\label{deduction} Let the lattice function $x(s)$
have the form
\[
x(s)=c_{1}q^{s}+c_{2}q^{-s}+c_{3}~\text{or}~x(s)=c_{1}s^{2}+c_{2}s+c_{3},
\]
where $q,c_{1},c_{2},c_{3}$ are constants. If $P_{n}[x_{k}(s)]$ is a
polynomial of degree $n$ in $x_{k}(s)$ with an arbitrary integer $k$, then
$\Delta_{k}P_{n}[x_{k}(s)]$ is a polynomial of degree $n-1$ in $x_{k+1}(s)$ .
\end{proposition}

We associate equation (\ref{adjoint equation}) with a difference equation of
the same type:
\begin{align}
\label{inhomogeneous equation}\sigma^{*}(s)\Delta_{-(n+1)}\nabla_{-n}
v(s)+\gamma(s,n)\Delta_{-n} v(s)+\eta(n) v(s)=P_{n-1}[x_{-n}(s)],
\end{align}
where $\gamma(s,n)$ and $\eta(n)$ are to be determined and $P_{n-1}%
[x_{-n}(s)]$ is an arbitrary polynomial of degree $n-1$ in $x_{-n}(s)$. We now
assume that
\begin{align*}%
\begin{split}
&  \Delta_{-n}^{(n)}\left[  \sigma^{*}(s)\Delta_{-(n+1)}\nabla_{-n}
v(s)+\gamma(s,n)\Delta_{-n} v(s)+\eta(n) v(s)\right] \\
&  =\sigma^{*}(s)\Delta_{-1}\nabla_{0} w(s)+\tau^{*}(s)\Delta_{0}
w(s)+\lambda^{*}w(s)=0
\end{split}
\end{align*}
and
\begin{align*}
w(s)=\Delta_{-n}^{(n)}v(s).
\end{align*}
Moreover, we assume that
\begin{align}
\label{exact}%
\begin{split}
&  \sigma^{*}(s)\Delta_{-n-1}\nabla_{-n} v(s)+\gamma(s,n)\Delta_{-n}
v(s)+\eta(n) v(s)\\
&  =\Delta_{-(n+1)}\left[  \sigma^{*}(s)\nabla_{-n} v(s)+\ell(s,n)v(s)\right]
\end{split}
\end{align}
with $\ell(s,n)$ to be determined.

Now, by the recurrence relations (\ref{recurrence1}) and (\ref{recurrence2}),
for any nonnegative integer $k$, we have
\begin{align*}
&  \gamma_{k}(s,n)=\frac{\sigma^{*}(s+k)-\sigma^{*}(s)+\gamma(s+k,n)\Delta
x_{-n}(s+k-\frac{1}{2})}{\Delta x_{-n+k-1}(s)},~\gamma_{0}(s,n)=\gamma(s,n),\\
&  \eta_{k}(n)=\eta(n)+\sum_{j=0}^{k-1}\Delta_{j-n}\gamma_{j}(s,n),~\eta
_{0}(n)=\eta(n).
\end{align*}
When $k=n$, we have
\begin{align}
\label{taustar1}%
\begin{split}
\tau^{*}(s)  &  =\gamma_{n}(s,n)=\frac{\sigma^{*}(s+n)-\sigma^{*}%
(s)+\gamma(s+n,n)\Delta x_{-n}(s+n-\frac{1}{2})}{\Delta x_{-n}(s+\frac{n-1}%
{2})}\\
&  =\frac{\sigma(s+1)-\sigma(s-1)}{\Delta x_{-1}(s)}-\tau(s-1)\frac{\nabla
x_{-1}(s)}{\Delta x_{-1}(s)},
\end{split}
\\%
\begin{split}
\label{lambdastar1}\lambda^{*}  &  =\eta(n)+\sum_{j=0}^{n-1}\Delta_{j-n}%
\gamma_{j}(s,n)\\
&  =\lambda-\nabla_{-1}\left(  \tau(s-1)\frac{\nabla x_{-1}(s)}{\nabla
x(s)}-\frac{\nabla\sigma(s)}{\nabla x(s)}\right)  .
\end{split}
\end{align}
From (\ref{taustar1}), we have
\begin{align}
\label{gamma}\gamma(s,n)=\frac{\sigma(s-n+1)-\sigma(s-1)-\tau(s-1)\nabla
x_{-1}(s)}{\Delta x_{-n}(s-\frac{1}{2})}.
\end{align}

Moreover, from (\ref{exact}) and Proposition \ref{product}, we obtain
\begin{align}
&  \ell(s+1,n)\frac{\Delta x_{-n}(s)}{\Delta x_{-(n+1)}(s)}+\Delta
_{-(n+1)}\sigma^{*}(s)=\gamma(s,n),\label{ellgamma}\\
&  \Delta_{-(n+1)}\ell(s,n)=\eta(n).
\end{align}
From (\ref{gamma}) and (\ref{ellgamma}), we have
\begin{align}
\label{ell}\ell(s,n)=\frac{\sigma(s-n)-\sigma(s-1)-\tau(s-1)\nabla x_{-1}%
(s)}{\nabla x_{-n}(s)}.
\end{align}
Then, $\eta(n)=\widehat{\lambda}$. Now, we compute $\lambda$. Note that, for
any nonnegative integer $k$,
\begin{align*}
\gamma_{k}(s,n)=\frac{\sigma(s+k-n+1)-\sigma(s-1)-\tau(s-1)\nabla x_{-1}%
(s)}{\Delta x_{k-n-1}(s)}=\widehat{\tau}_{k}(s).
\end{align*}
From (\ref{muhat}) and (\ref{lambdastar00}),
\begin{align*}
\lambda_{n}=\nabla_{-1}\tau_{-1}(s)+\widehat{\lambda}+\sum_{j=0}^{n-1}%
\Delta_{j-n}\widehat{\tau}_{k}(s).
\end{align*}
Then by (\ref{lambdastar1}), we have $\lambda=\lambda_{n}$.

By Proposition \ref{deduction}, we have
\begin{align}
\label{simple}\sigma^{*}(s)\nabla_{-n} v(s)+\ell(s,n)v(s)=P_{n}[x_{-(n+1)}%
(s)].
\end{align}
Now, we solve the equation (\ref{simple}). First, we consider the solution of
the following equation:
\begin{align}
\label{homegeneous0}\sigma^{*}(s)\nabla_{-n} v(s)+\ell(s)v(s)=0.
\end{align}
The equation of (\ref{homegeneous0}) can be rewritten as
\begin{align}
\label{homegeneous1}\sigma(s-n)v(s)=\left(  \sigma(s-1)+\tau(s-1)\nabla
x_{-1}(s)\right)  v(s-1).
\end{align}
Moreover, from Pearson type equation (\ref{pearson1}), we have
\begin{align*}
\rho(s)\sigma(s)=\left[  \sigma(s-1)+\tau(s-1)\nabla x_{-1}(s)\right]
\rho(s-1).
\end{align*}

Then the equation (\ref{homegeneous1}) becomes
\begin{align*}
\frac{v(s)}{v(s-1)}=\frac{\rho(s)\sigma(s)}{\rho(s-1)\sigma(s-n)}.
\end{align*}
It is easy to see that the solution of the above equation is
\begin{align*}
v(s)=C\rho(s)\prod_{i=0}^{n-1}\sigma(s-i),
\end{align*}
Where $C$ is any constant.

Now, let
\begin{align}
\label{solution}v(s)=C(s)\rho(s)\prod_{i=0}^{n-1}\sigma(s-i).
\end{align}
Then, substituting (\ref{solution}) into (\ref{simple}), we have
\begin{align*}
\nabla_{-n}C(s)=\frac{P_{n}[x_{-(n+1)}(s)]}{\rho(s)\prod_{i=0}^{n}\sigma
(s-i)}.
\end{align*}
Then, using Proposition \ref{integration},
\begin{align*}
C(s)=\int_{N}^{s}\frac{P_{n}[x_{-(n+1)}(t)]}{\rho(t)\prod_{i=0}^{n}%
\sigma(t-i)}d_{\nabla}x_{-n}(t)+\widetilde{C},
\end{align*}
where $\widetilde{C}$ is any constant. Since
\begin{align*}
\rho(s)y(s)=w(s)=\Delta_{-n}^{(n)}v(s),
\end{align*}
we obtain:

\begin{theorem}
\label{extension2} If
\[
\lambda=\lambda_{n} ~\text{and}~ \lambda_{m}\neq\lambda_{n}~\text{for}%
~m=0,1,\dots,n-1,
\]
then the general solution of (\ref{operator}) is
\begin{align*}%
\begin{split}
y_{n}(s)  &  =\frac{C}{\rho(x)}\Delta_{-n}^{(n)}[\rho(s)\prod_{j=0}%
^{n-1}\sigma(s-j)]\\
&  +\frac{1}{\rho(x)}\Delta_{-n}^{(n)}\left[  \rho(s)\prod_{j=0}^{n-1}%
\sigma(s-j)\int_{N}^{s}\frac{P_{n}[x_{-(n+1)}(t)]}{\rho(s)\prod_{j=0}%
^{n}\sigma(t-j)}d_{\nabla}x_{-n}(t)\right]  ,
\end{split}
\end{align*}
where $C$ is any constant and $P_{n}(\cdot)$ is an $n$-th polynomial.
\end{theorem}

\begin{ack}
The authors were supported by the Fundamental Research Funds for the central
Universities, grant number 20720150006 and Natural Science Foundation of
Fujian province, grant number 2016J01032.
\end{ack}


\end{document}